\documentclass[fleqn,12pt]{article}
\usepackage{graphicx}
\usepackage{latexsym}
\usepackage{amsmath}
\usepackage{amssymb}
\usepackage{amsthm}
\usepackage{hyperref}

\def \C {{\mathbb C}}

\def \A {{\mathcal A}}

\def \S {{\mathcal S}}
\def \c {{\mathcal{CC}}}
\def \V {{\mathcal V}}
\def \Ll {{\mathcal L}}

\def \la {{\lambda}}

\newtheorem{theorem}{Theorem}[section]

\newtheorem{cor}[theorem]{Corollary}
\newtheorem{lemma}[theorem]{Lemma}

\newtheorem{obs}[theorem]{Observation}

\newtheorem{definition}{Definition}[section]

\newtheorem{rem}[theorem]{Remark}

\author{}

\title{Pattern polynomial graphs}
\author{A. Satyanarayana Reddy\footnote{Department of Mathematics and Statistics, Indian
Institute of Technology, Kanpur, India 208016; (e-mail:
satya@iitk.ac.in).} \and Shashank K Mehta \footnote{Department of Computer Science and
Engineering, Indian Institute of
Technology, Kanpur, India 208016; (e-mail: skmehta@cse.iitk.ac.in). This work was partly
supported by Research-I
Foundation, IIT-Kanpur.}}
\date{}
\begin{document}
\maketitle
\begin{abstract}
A graph $X$ is said to be a pattern polynomial graph if its adjacency algebra
is a coherent algebra. In this study we will find a necessary and sufficient condition
 for a graph to be a pattern polynomial graph.
 Some of the properties of the graphs which are polynomials in the  pattern polynomial graph
 have been studied. We also identify known graph classes which are pattern polynomial graphs.
\end{abstract}

\section{Introduction and preliminaries}

Let $M_n(\C)$ denote the set of all $n\times n$ matrices over the
field of complex numbers $\C$.  Let $\C[A]$ denote the set of all matrices which are
polynomials in $A$ with coefficients from  $\C$. Clearly $\C[A]$ is
an algebra over $\C$ for any $A\in M_n(\C)$. The dimension of $\C[A]$ over $\C$ as a vector
space, is the degree of the minimal polynomial of $A$. If $A$ is
diagonalizable, then  from the  following lemma, it's dimension is
equal to the number of distinct eigenvalues of $A$.

\begin{lemma}[Hoffman $\&$ Kunze~\cite{H:K}] \label{lem:aa3}
A matrix is diagonalizable if and only if its minimal polynomial has
all distinct linear factors over $\C$.
\end{lemma}

\begin{definition} Hadamard product of two $n\times n$
matrices $A$ and $B$ is denoted by  $A\odot B$ and is defined as
$(A\odot B)_{xy}=A_{xy}B_{xy}$.
\end{definition}
 Two $n\times n$ matrices $A$ and $B$ are said to be {\em disjoint}
if their Hadamard product is the zero matrix.

\begin{definition}
A subalgebra of $M_n(\mathbb{C})$ is called {\em coherent} if it contains
the matrices $I$ and $J$ and if it is closed under conjugate-transposition
and Hadamard multiplication. Here $J$ denotes the matrix with every entry being $1$.
\end{definition}

\begin{theorem}\label{thm:unique}\cite{B:R}
Every coherent algebra contains unique basis of mutually disjoint
$0,1$- matrices (matrices with entries either $0$ or $1$).
\end{theorem}

We call this unique basis of mutually disjoint 0-1 matrices as a {\emph{standard
basis}}.

\begin{cor}\label{cor:sb}
Every $0,1$-matrix in a coherent algebra is the  sum of one or more matrices in its standard basis.
\end{cor}
\begin{proof}
Let $\mathcal{M}$ be a coherent algebra over $\C$ with its
standard basis $\{M_1,\ldots M_t\}$. Let $B\in \mathcal{M}$ be a
$0,1$-matrix, then $B=\sum_{i=1}^ta_iM_i$ where $a_i\in \C$.
$B=B\odot B=\sum_{i=1}^ta_i^2M_i\Rightarrow a_i^2=a_i$. Hence the
result follows.
\end{proof}

\begin{cor}\label{cor:leq}
If $\mathcal{M}$ is a commutative coherent algebra over $\C$, then $\dim(\mathcal{M})\leq n$.
\end{cor}

\begin{proof}
Since $J$ commutes with every element in $\mathcal{M}$. Hence all the row(column) sums of every matrix in $\mathcal{M}$ are equal. Consequently the number of elements in the standard basis of $\mathcal{M}$ is atmost $n$, as the matrices in the standard basis of $\mathcal{M}$  are disjoint and whose sum is $J$.
\end{proof}

\begin{obs} \label{obs:int}
The intersection of coherent algebras is again a coherent algebra.
\end{obs}

\begin{definition}
 Let $A\in M_n(\C)$,  then coherent closure of $A$, denoted
by $\langle \langle A \rangle \rangle$ or $\c(A)$, is the smallest
coherent algebra containing $A$.
\end{definition}

In  Section~\ref{sec:one} we will see a  necessary and sufficient
condition for any matrix $A\in M_n(\C)$ such that $\C[A]=\c(A)$. In
the remaining sections we consider $A$ to be  the adjacency matrix
of a graph $X$. If $A$ is the adjacency matrix  of a graph $X$ and
$\C[A]=\c(A)$, then $X$ will be called a pattern polynomial graph. Some of
the  properties of pattern polynomial graphs are given in 
Section~\ref{sec:two}. In  Section~\ref{sec:three} we will see  a
few graph classes which are pattern polynomial graphs. Few partially
balanced incomplete block designs from pattern polynomial graphs
are constructed in Section~\ref{sec:four}.   Properties of graphs
which are polynomials in the pattern polynomial graph are provided
in the Section~\ref{sec:five}.

\section{$\C[A]=\c(A)$}\label{sec:one}

In this section we will see a  necessary and sufficient condition for a matrix $A\in M_n(\C)$ such that $\C[A]=\c(A)$. For that we construct a vector space which  lies in between  $\C[A]\;and\;\c(A)$ as follows.

Let $\ell$ be  the degree of the minimal polynomial of $A$. Then
$\{I,A,\ldots, A^{\ell-1}\}$ is a basis for $\C[A]$ over $\C$.
Let $\textbf{y}=(y_0,y_1\ldots, y_{\ell-1})\in \C^\ell$ be a variable and   $$B(\textbf{y})=y_0I+y_1A+\dots y_{l-1}A^{\ell-1}=
\begin{bmatrix}
    p_{11}(\textbf{y}) & p_{12}(\textbf{y})& \dots  & p_{1n}(\textbf{y}) \\
    p_{21}(\textbf{y}) & p_{22}(\textbf{y})& \dots & p_{2n}(\textbf{y}) \\
        \vdots &\vdots & \ddots &   \vdots \\
        p_{n1}(\textbf{y}) & p_{n2}(\textbf{y})  & \dots & p_{nn}(\textbf{y})
     \end{bmatrix}.$$
Where $p_{ij}(\textbf{y})=p_{ij}(y_0,y_1\ldots, y_{\ell-1})$ is a polynomial in the variables  $y_0,y_1,\ldots, y_{\ell-1}$.
Let us assume $\{q_1(\textbf{y}),q_2(\textbf{y}),\ldots, q_r(\textbf{y})\}$ be the set of distinct polynomials in the matrix $B(\textbf{y})$.
We define the matrices called pattern matrices of $A$, as $P_j\;(1\leq j\leq r)$ as
$$(P_j)_{s,t}=\begin{cases} 1, & {\mbox{ if }} B(\textbf{y})_{s,t}=q_j(\textbf{y}), \\
 0, &  {\mbox{otherwise. }}
\end{cases}$$
Let $\Ll (A)=L\{P_1,P_2,\ldots,P_r\}$
denote the linear span of matrices $P_1,\ldots,P_r$. Then  $\Ll (A)$
is a subspace of $M_n(\C)$. From the definition of $\Ll(A)$ we have the following observation.

\begin{obs}\label{obs:smallest}
\begin{enumerate}
\item $P_i\odot P_j=0\;\forall i,j \;1\leq i,j \leq r$, $\sum_{i=1}^rP_i=J\in \Ll(X)$, $I\in \Ll(A)$.
\item $\Ll(A)$ is closed under Hadamard product.
\begin{proof}
Let $M,N\in \Ll(A)$, so $M=\sum_{i=1}^ra_iP_i\;, N=\sum_{i=1}^rb_iP_i$ where  $a_i,b_i\in \C$.  Then  $M\odot N=\sum_{i=1}^ra_ib_iP_i\in\Ll(A)$.
\end{proof}
\item $\Ll(A)$ is the smallest subspace of $M_n(\C)$ closed under Hadamard product and contains all powers of $A$. Consequently $\C[A]\subseteq \Ll(A) \subseteq \c(A)$ and $l\leq r$.
\item If $P_i^T\in \{P_1,P_2,\ldots,P_r\} \;for\;all\; 1\leq i \leq r$, then $\Ll(A)$ is also closed under conjugate transposition. In particular, if $A$ is symmetric, then all pattern matrices are symmetric hence $\Ll(A)$ is closed under conjugate transposition.
\end{enumerate}
\end{obs}

With the above observation, we are providing the main result of this section.

\begin{theorem}
Let $A\in M_n(\C)$ be a symmetric matrix. Then $\C[A] = \c(A)$   if and only if $\ell=r$.
\end{theorem}

\begin{proof}
If  $\C[A] = \c(A)$, then $\C[A]= \Ll(A)$ hence $\ell=r$. Conversely
suppose  $\ell=r$ then $\C[A]= \Ll(A)$. In particular $\Ll(A)$ is closed under ordinary multiplication. Hence from the above obseravtion $\Ll(A)$ is a coherent algebra but $\c(A)$ is the smallest coherent algebra containing $\C[A]$. Hence the result follows.
\end{proof}

\section{Pattern polynomial graphs}\label{sec:two}

From now onwards  we suppose that $A$ (or A(X)) is the adjacency
matrix of a graph $X$, hence $A$ is a symmetric $0,1$ matrix. We denote
$\C[A]$ by $\A(X)$, $\c(A)$ by $\c(X)$ and $\Ll(A)$ by $\Ll(X)$. We
call $\A(X)$ as the adjacency algebra of the  graph $X$, $\c(X)$ as
coherent closure of $X$.  First we will provide a few results on
adjacency algebra of a graph. Then we will see some additional
properties of the vector space $\Ll(X)$.  For two vertices $u$ and
$v$ of a connected graph $X$, let $d(u,v)$ denote the length of the
shortest path from $u$ to $v$. Then the diameter of a connected
graph $X=(V,E)$ is $\max\{d(u,v): \; u,v \in V\}$. It is shown in
Biggs~\cite{biggs} that if $X$ is a connected graph with diameter
$d$, then
\begin{equation}\label{eq:ell}
 d+1\leq \dim(\A(X)) \leq n
\end{equation}
where $\dim(\A(X))$ is the dimension of $\A(X)$ as a vector space
over $\C$.

A graph $X_1=(V(X_1), E(X_1))$ is said to be {\em isomorphic} to the
graph $X_2=(V(X_2), E(X_2))$, written $X_1\cong X_2$, if there is a
one-to-one correspondence $\rho: V(X_1)\rightarrow V(X_2)$ such that
$\{v_1, v_2\}\in E(X_1)$ if and only if  $\{\rho(v_1),
\rho(v_2)\}\in E(X_2)$. In such a case, $\rho$ is called an {\em
isomorphism} of $X_1$ and $X_2$. An isomorphism of a graph $X$ onto
itself is called an {\em automorphism}. The collection of all
automorphisms of a graph $X$ is denoted by $\text{Aut}(X)$. It is
well known   that $\text{Aut}(X)$ is a group under composition of
two maps. It is easy to see that $\text{Aut}(X) = \text{Aut}(X^c)$,
where $X^c$ is the complement of the graph X. If $X$ is a  graph
with $n$ vertices we can think of $\text{Aut}(X)$ as a subgroup of
$\S_n$. Under this correspondence, if a graph $X$ has $n$ vertices
then $\text{Aut}(X)$ consists of $n \times n$ permutation matrices
and for each $g\in \text{Aut}(X)$,  $P(g)$ will denote the
corresponding permutation matrix. The next result gives a method to
check whether a given permutation matrix is an element of
$\text{Aut}(X)$ or not.

\begin{lemma}[Biggs~\cite{biggs}] \label{lem:aut}
Let A be the adjacency matrix of a graph $X$. Then $g\in
\text{Aut}(X)$ is an automorphism of $X$ if and only if $P(g)A=A
P(g)$.
\end{lemma}

The following result is very useful in this work, which provides the necessary and sufficient condition for a graph to be connected and  regular.

\begin{lemma}[Biggs~\cite{biggs}]\label{lem:J}
A graph $X$ is connected  regular if and only if $J\in \A(X)$.
\end{lemma}

\begin{cor}\label{cor:J}
 If $X$ is a regular graph then $J$ is polynomial in either $A$ or $A^c$.
\end{cor}

\begin{proof}
For every graph $X$, either $X$ or $X^c$ is connected. Hence the result follow from the above Lemma.
\end{proof}
Let $X$ be a connected regular graph. Then from the above lemma we have $A(X^c)\in \A(X)$. Hence we have the following corollary.

\begin{cor}
 Let $X$ be a connected regular graph. Then $X^c$ is connected if and only if
$\A(X)=\A(X^c)$.
\end{cor}

\begin{definition}
A graph $X$ is said to be a {\emph{pattern polynomial graph}} if its pattern matrices are polynomials in the adjacency matrix of $X$.
\end{definition}

\begin{lemma}\label{lem:pp}
A graph $X$ is a pattern polynomial graph if and only if
$\A(X)=\c(X)$.
\end{lemma}
Consequently if $X$ is a  pattern polynomial graph, then $\A(X)=\Ll(X)=\c(X)$.
For any graph $X$, we have $\Ll(X)=\Ll(X^c)$ and $\c(X)=\c(X^c)$. Hence we have the following result.

\begin{cor}\label{cor:comp}
Let $X$ be a pattern polynomial graph and $X^c$ is also connected. Then $\A(X)=\Ll(X)=\c(X)=\A(X^c)=\Ll(X^c)=\c(X^c)$.
\end{cor}


\subsection{Properties of pattern polynomial graphs}

In this subsection we prove that if $X$ is a pattern polynomial graph, then $X$ is necessarily a \\
a) connected regular graph b) distance polynomial graph c) walk
regular graph d) strongly distance-balanced graph. e) edge regular
graph whenever $A$ is also a pattern matrix. We also show that
 every pattern polynomial graph except $K_2$
(complete graph with 2 vertices) has at least one  multiple eigenvalue. In
particular, if $X$ is a pattern polynomial graph with odd number of
vertices, then we show that $\dim(\A(X))\leq \frac{n-1}{2}$.
Throughout this section we suppose that $X$ is a pattern polynomial
graph that is $\A(X)=\Ll(X)=\c(X)$ or $\ell=r$. From Lemma~\ref{lem:J} we have the following result.

\begin{lemma}
Every pattern polynomial graph is a connected regular graph.
\end{lemma}
It is easy to see that the  converse  of the above lemma is not true.
The above  lemma provides a necessary condition to have $\A(X)=\Ll(X)$.
That is if the graph  $X$ is  either not  connected or not regular, then  $\A(X)\subsetneq \Ll(X)$.
  Now we will state another necessary condition stronger than this.\\

\noindent\textbf{Distance polynomial graph}\\

Let $X =
(V,E)$ be a connected graph with diameter $d$. The {\emph{$k$-th distance matrix}} $A_k (0\le k \le d)$ of $X$,  is defined as
$ (A_{k})_{rs} = \left\{ \begin{array}{cl} 1, & {\mbox{ if }} d(v_{r},v_{s}) = k \\
0, & {\mbox{ otherwise.}} \end{array} \right. $ \\  It follows that
$$
A_0=I\;\mbox{(Identity matrix)}, \;A_1=A,\; A_0+A_1+\dots +A_d=J.$$
A  connected graph $X$ of diameter $d$ is said to be a distance polynomial graph  if   $A_k\in \A(X) \;for\;0\le k \le d$. From the Lemma~\ref{lem:J}, the following observation is evident.

\begin{obs}\label{obs:dp}
Every distance polynomial graph is a regular connected graph. The converse is generally not true but every regular connected graph of diameter $2$ is distance polynomial.
\end{obs}

\begin{lemma}\label{lem:dp}
If $X$ be a connected graph of diameter $d$, then $A_k\in \Ll(X) \;(0\le k \le d)$, where $A_k$ is the $k$-th distance matrix of $X$.
\end{lemma}

\begin{proof}
We prove the result by induction on $d$.  $A_0(=I),A_1(=A)\in \A(X)\subseteq \Ll(X)$. So the result is true for $d=1$. Suppose that $A_0,A_1,\ldots, A_{s-1}\in \Ll(X)$. Then $J-I-A_0-A_1-\dots- A_{s-1}\in \Ll(X)$. Let $M=A^s\odot (J-I-A_1-\dots -
A_{s-1}) \in \Ll(X)$. Observe that $M_{ij}\neq 0$ if and only if
$(A_s)_{ij}=1$. As $M\in \Ll(X)$ we have
\begin{equation}\label{eq:M}
M=\sum_{i=1}^r a_iP_i\; \mbox{where}\; a_i\in \C.
\end{equation}
Hence $A_s=\sum_{i:a_i\ne 0}P_i\in \Ll(X)$.
\end{proof}
Now from Corollary~\ref{cor:sb}, every distance matrix is
the sum of one or more pattern matrices.

\begin{cor}\label{cor:ppdp}
Every pattern polynomial graph is a distance polynomial graph.
\end{cor}
From Observation~\ref{obs:dp}, every connected regular graph of diameter $2$ is distance polynomial. But all connected regular graphs of diameter $2$ are not pattern polynomial graphs.

\begin{definition}[Paul M.Weichsel~\cite{P:W1}]
Let $v$ be a vertex in the graph $X$ of diameter $d$. The
generalized degree of $v$ is the $d$-tuple $(k_1,k_2\ldots,k_d)$, where $k_i$ is the number of vertices whose distance from $v$ is $i$.
 The graph G is called {\emph{super-regular}} if each vertex has the same generalized degree.
\end{definition}

\begin{theorem}[Paul M.Weichsel~\cite{P:W1}]
Let $X$ be a connected graph of diameter $d$. $X$ is super-regular graph if and only if
$(A_iA_j)_{rs}=(A_jA_i)_{rs}$ for all $r,s$.
\end{theorem}

From the Corollary~\ref{cor:ppdp} every pattern polynomial graph is a super regular graph.
In the present literature super-regular graphs are also called 
distance-degree regular or strongly distance-balanced graphs for
details refer~\cite{S:P}.

\noindent\textbf{Walk-regular graph}\\
A graph $X$ is said to be
walk-regular if for each $s$, the number of closed walks of length $s$ starting at
a vertex $v$ is independent of the choice of $v$.

\begin{theorem}\cite{G:M}\label{thm:walk}
Let $A$ be the adjacency matrix of a  graph $X$.
Then  $X$ is walk-regular if and only if  the diagonal entries of $A^s\;\forall s$ are all equal.
\end{theorem}

The following lemma and corollary can be obtained from the Corollary~\ref{cor:leq} and above theorem. 

\begin{lemma}
Every pattern polynomial graph is a walk-regular graph.
\end{lemma}

\begin{cor}\label{cor:regular}
If $X$ is pattern polynomial graph then every pattern matrix other than the identity  matrix  is the adjacency matrix of a regular graph.
\end{cor}

\begin{rem}
Let $X$ be a pattern polynomial graph and $P\in\A(X)$ be a permutation matrix. Then from Corollary~\ref{cor:sb} and from the  above result $P$ is an element in the standard basis of $\A(X)$. Further it is easy to see that the set of all permutation matrices in $\A(X)$ forms  an elementary abelian 2-group since matrices in $\A(X)$ are symmetric.
\end{rem}

Let $X$ be a pattern polynomial graph  and $\{I=P_1,P_2,\ldots,
P_r\}$ be the set of its pattern matrices. Let us call the graph  $X_{P_i} \;2\leq
i\leq r$ as pattern graph of $X$  with adjacency matrix $P_i$. Then form
Lemma~\ref{lem:aut}, we have   $\text{Aut}(X) \subseteq
\text{Aut}(X_{P_i})$. In fact in the next section we will show that $\text{Aut}(X) \subseteq
\text{Aut}(X_{P_i})$ is true even if $X$ is not a pattern polynomial
graph. Now we will show that
every pattern polynomial graph except $K_2$ has at least one multiple
eigenvalue. In order to prove this, we need the following definition
and the result.

\begin{definition}
A graph is said to be vertex
transitive if its automorphism group acts transitively on $V$. That is for any
two vertices $x,y \in V, \exists g \in G$ such that $g(x)=y$.
\end{definition}

\begin{lemma}\label{H}[Biggs~\cite{biggs}]
Let $X$ be a $k$-regular vertex transitive graph, and $\lambda$ be a simple eigenvalue of $X$. Then $$\lambda = \left\{ \begin{array}{cl} k, & {\mbox{ if }}\; |V|\; \mbox{is odd}, \\
\mbox{one of the integers}\;2\alpha-k\;(\;0\leq \alpha \leq k),\; &  {\mbox{ if}}\;|V|\; \mbox{is even.} \end{array} \right. $$

\end{lemma}

\begin{cor}
If $X$ is a vertex transitive graph and $X\ne K_2$, then $X$ has at least one  multiple eigenvalue.
\end{cor}

\begin{cor}
If $X$ is a pattern polynomial graph and $X\ne K_2$, then $X$ has a multiple eigenvalue.
\end{cor}

\begin{proof}
If  all eigenvalues of $X$ are simple, then $\dim (\A(X))=n$ so
every pattern matrix is a symmetric permutation matrix  whose sum
is $J$. Consequently $X$ is a vertex transitive graph.
\end{proof}

We can extend the above result to arbitrary graphs in the following manner.

\begin{cor}\label{cor:mul}
Let $X$ be a graph with $n>2$ vertices and $\c(X)$ is a commutative
algebra. Then $\dim(\A(X))\leq n-1$ and  $\dim(\c(X))\leq n$.
\end{cor}
\begin{proof}
First observe that $X$ has at least one multiple eigenvalue if and only if $\dim(\A(X))\leq n-1$. Now from the Corollary~\ref{cor:leq}, we have  $\dim(\c(X))\leq n$. If $\dim(\A(X))= n$, then we will get a contradiction from above corollary.
\end{proof}

 If $X$ is a pattern polynomial graph with odd number of vertices,
then from Corollary~\ref{cor:regular} we have stronger result than
above.

\begin{lemma}
If $X$ is a pattern polynomial graph with odd number of vertices, then $\dim(\A(X))\leq \frac{n+1}{2}$.
\end{lemma}

\begin{proof}
First observe that  $\dim(\A(X))$ is the number of pattern matrices. From Corollary~\ref{cor:regular}, all  pattern graphs of $X$ are regular with odd number of vertices. Consequently each  is an even regular graph with regularity $\geq 2$. Hence $X$ has atmost $\frac{n-1}{2}$ pattern graphs.
\end{proof}

A graph is said to be an {\em edge-regular} graph if any two of its
adjacent vertices have the same number of common neighbours. The following result is easy to see.

\begin{lemma}
If $X$ is a pattern polynomial graph and its adjacency matrix itself is a pattern matrix
then $X$ is an edge-regular graph.
\end{lemma}

\begin{proof}
Let $X$ be a pattern polynomial graph with adjacency matrix $A$. Let $\{P_0=I,P_1=A,P_2,\ldots, P_r\}$ be the standard basis of $\A(X)$. Hence $A^2=a_0I+a_1A+a_2P_2+\dots +a_rP_r$ where $a_i\in \C$. 
Consequently any two adjacent vertices have exactly $a_1$ common neighbours.
\end{proof}

In the following section, we see few classes of graphs which satisfy
the condition $\ell=r$. Consequently they are pattern polynomial graphs

\section{Some graph classes which are pattern polynomial} \label{sec:three}

In this section we will prove that the following classes of graphs are
pattern polynomial graphs a)orbit polynomial graphs b) distance
regular graphs hence distance transitive graphs c) connected compact
regular graphs.
\begin{definition}
Let $G$ be a subset of $n\times n$ permutation matrices forming a
group. Then $\V_{\C}(G)=\{A\in M_n(\C): PA=AP\; \forall P\in G\}$
forms an algebra over $\C$ called the {\em centralizer algebra} of
the group $G$.
\end{definition}

\begin{definition}
If G is a group acting on a set  $V$, then $G$ also acts on
$V\times V$ by $g(x,y)=(g(x),g(y))$. The orbits of G on $V\times
V$ are called {\em orbitals}. In the context of graphs, the
orbitals of graph $X$ are orbitals of its automorphism group
$\text{Aut}(X)$ acting on the vertex set of $X$.
That is, the orbitals are the orbits of the
arcs/non-arcs of the graph $X=(V,E)$. The number of orbitals is
called the {\em rank} of $X$.
\end{definition}

An orbital can be represented by a $0,1$-matrix $M$ where $M_{ij}$
is $1$ if $(i,j)$ belongs to the orbital. We can associate
directed graphs to these matrices. If the matrices are symmetric,
then these can be treated as undirected graphs.

\begin{obs}
\begin{itemize}
\item The `1' entries of any orbital matrix are either all on the diagonal or all are off diagonal.
\item The orbitals containing $1$'s on the diagonal will be called {\em diagonal} orbitals.
\end{itemize}
\end{obs}

\begin{definition}
The centralizer algebra of a graph $X$ denoted by $\V(X)$ is the centralizer algebra of its automorphism group acting on the vertex set of $X$.
\end{definition}

\begin{theorem}\cite{K:R:R:T}
$\V_(X)$ is a coherent algebra and orbitals of
$\text{Aut}(X)$ acting on the vertex set of $X$ form its unique 0-1 matrix basis.
\end{theorem}

Since $\text{Aut}(X)=\text{Aut}(X^c)$,  we have  $\V(X)=V(X^c)$. Also
$\c(X)$ is the smallest coherent algebra containing $A(X)$ and
$\V(X)$ is a coherent algebra of $X$ containing $A(X)$ so
$\c(X)\subseteq \V(X)$.

\paragraph{Orbit polynomial graphs}

\begin{definition}
A graph $X=(V,E)$ is {\em orbit polynomial} graph if each orbital matrix is a member of $\A(X)$. That is, each orbital matrix is a polynomial in $A$.
\end{definition}

\begin{lemma}
If $X$ is orbit polynomial graph if and only if $\A(X)=\V(X)$
\end{lemma}
If $X$ is a orbit polynomial graph, then $\A(X)=\c(X)=\V(X)$. Hence
every orbit polynomial graph is a pattern polynomial graph.

If $X$ is an orbit polynomial graph and $X^c$ is connected, then
 from above lemma we have
$\A(X)=\A(X^c)=\c(X)=CC(X^c)=\V(X)=V(X^c)$. So we have the
following result.

\begin{cor}
If $X$ is an orbit polynomial graph and $X^c$ is connected then
$X^c$ is also a  orbit polynomial graph.
\end{cor}

For any graph $X$, we have $\A(X)\subseteq \Ll(X)\subseteq
\c(X)\subseteq \V(X)$. Consequently from Corollary~\ref{cor:sb}
every pattern matrix is the sum of one or more orbital matrices. Further
by definition,  orbital matrices commute with all automorphisms of
$X$ hence we have the following result.

\begin{lemma}
Let $X$ be any graph and $\{P_1,P_2,\ldots, P_r\}$ be the set of all
pattern matrices of adjacency matrix of $X$. Then $\text{Aut}(X)
\subseteq \text{Aut}(X_{P_i})\; 1\leq i\leq r $.
\end{lemma}

Every connected vertex transitive graph of prime order is an orbit polynomial graph, see [Beezer~\cite{R:B1}]. Following lemma gives a stronger result.

\begin{lemma}
If $X$ is a connected graph of prime order, then $X$ is orbit polynomial graph if and only if $\V(X)$ is commutative.
\end{lemma}

\begin{proof}
If $X$ is an orbit polynomial graph then clearly $\V(X)$ is
commutative. Conversely suppose that $\V(X)$ is commutative, then
the identity matrix  is in the standard basis of $\V(X)$.
Consequently  $X$ is a vertex transitive graph, hence the result
follows.
\end{proof}

An easy consequence of this lemma is that if $X$ is a connected graph of prime order and $\V(X)$ is commutative, then $X$ is a pattern polynomial graph.

\paragraph{Distance transitive graphs}

\begin{definition}
A graph $X$ is {\emph{distance
transitive}} if for all vertices $u,v,x,y$ of $X$ such that $d(u,v)=d(x,y)$
then there is a $g$ in $\text{Aut}(X)$ satisfying $g(u)=x$ and $g(v)=y$.
\end{definition}

\begin{rem}\label{rem:orb}
From the definition of distance transitivity following facts are
immediate: a) For a distance transitive graph with diameter $d$,
distance matrices and orbital matrices coincide, consequently its
rank is $d+1$. b) From Equation~\ref{eq:ell} and the fact that $\A(X)\subseteq \V(X)$, if $X$ is a distance transitive graph with
diameter $d$, then dimension of $\A(X)$ is $d+1$. Further orbital
matrices form a basis for $\A(X)$. This also implies the following lemma.
\end{rem}

Now the following result is immediate from the Theorem~\ref{thm:dis} and the
above Remark.

\begin{lemma}\label{lem:dis}
Every distance transitive graph is an orbit polynomial graph.
\end{lemma}

Converse of the above lemma is not true, as every  vertex transitive
graph of prime order is not a distance transitive graph.
so we have \\
Distance transitive graph $\Rightarrow$ Orbit polynomial graph
$\Rightarrow$ Pattern polynomial graph $\Rightarrow$ Distance
polynomial graph.

\paragraph{Compact graphs}

\begin{definition}
A graph $X$ is said to be  {\emph{compact}} if every doubly stochastic matrix which commutes with $A(X)$ is a convex combination of matrices from $\text{Aut}(X)$.
\end{definition}

A permutation group on a set X is generously transitive if, given any two
points, there is a permutation which interchanges them.

\begin{theorem} \label{lemma:compact}\cite{G:D1}
Let $X$be a connected regular graph with $r$
distinct eigenvalues. If $X$is compact, then $\text{Aut}(X)$ is a generously transitive permutation group with rank $r$.
\end{theorem}
In a compact connected regular graph $X$ the number of distinct eigenvalues of $A(X)$ is same as the number of orbitals. It is also the dimension of $\A(X)$. Hence we have the following corollary from the fact that $\A(X)\subseteq \V(X)$.
\begin{cor}
Every compact connected regular graph is orbit polynomial graph.
\end{cor}

Godsil~\cite{G:D1} showed that if $n\geq 7$, then the line graph of the complete graph $K_n$, is a distance transitive graph  but not compact graph.  It is also easy to check that if $X$ is compact, then so is its compliment $X^c$ . Hence if $X$ is connected compact regular graph and $X^c$ is also connected, then $X^c$ is compact connected regular graph but it need not be a distance transitive graph $X=C_6$(the cycle graph) is such an example.

Now we will see class of graphs which are pattern polynomial graphs but  need not be orbit polynomials graphs.

\paragraph{Distance regular Graphs}

\begin{definition}
A connected graph is distance regular if for any two vertices $u$ and $v$, the
number of vertices at distance $i$ from $u$ and $j$ from $v$ depends only on $i$, $j$, and
the distance between $u$ and $v$. These graphs are necessarily regular, since $u$
may be equal to $v$.
\end{definition}

It is easy to see that every distance transitive graph is distance
regular. In fact, there are many distance regular graphs whose automorphism
group is trivial \cite{spence}. The following theorem establishes that  every distance
regular graph is a  pattern polynomial graph.

\begin{theorem}\label{thm:dis}[Damerell~\cite{dam}]
Let $X$ be a distance regular graph with diameter $d$.  Then
$\{A_0,A_1,\dots ,A_d\}$ is a basis for the adjacency algebra
$\A(X)$, and consequently the dimension of $\A(X)$ is $d+1$.
\end{theorem}

\begin{cor}\label{cor:A=CC}
Every distance regular graph is a pattern polynomial graph.
\end{cor}
Observe that if $X$ is a distance regular graph then $A$ itself is a pattern matrix. Consequently if $X$ is a distance regular graph with diameter $\geq 3$, then $X^c$ is not distance regular. Now if $X$ is
a distance regular graph and $X^c$ is connected then  from Corollary~\ref{cor:comp}
$\A(X)=\A(X^c)=\c(X)=\c(X^c)$. Hence $X^c$ is also  pattern  polynomial
graph but it need not be a distance regular graph for example $C_6^c$. Further there are
distance regular graphs whose automorphism group is trivial \cite{spence}. So they
can't be orbit polynomial graphs. Finally, if $X$ is a distance
regular graph with diameter $\geq 3$ with trivial automorphism group
and $X^c$ is connected, then $X^c$ is neither distance regular nor
orbit polynomial graph, but it is a pattern polynomial graph.

The following  diagram gives the relationship among some of the
graph classes which we studied in this work.
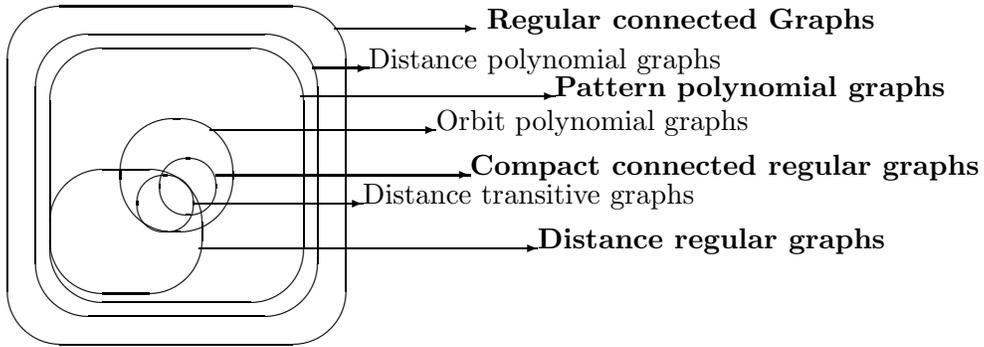
\vspace*{10mm}

 \begin{figure}[tbh]

\unitlength=0.75mm
 \begin{picture}(0,0)
 \put(35,-15){\oval(60,60)}
\put(63,11){\vector(1,0){25}}
\put(90,11){\small{\textbf{Regular connected  Graphs}}}

\put(35,-15){\oval(50,50)}
\put(59,4){\vector(1,0){10}}
\put(69,4){\small{Distance polynomial graphs}}

\put(35,-15){\oval(45,45)}
\put(57,-1){\vector(1,0){45}}
\put(102,-1){\small{\textbf{Pattern polynomial graphs}}}

\put(35,-15){\oval(20,20)}
\put(41,-7){\vector(1,0){40}}
\put(81,-7){\small{Orbit polynomial graphs}}

\put(33,-20){\oval(10,10)}
\put(42,-15){\vector(1,0){45}}
\put(87,-15){\small{\textbf{Compact  connected regular graphs}}}

\put(37,-17){\oval(10,10)}
\put(38,-20){\vector(1,0){30}}
\put(68,-20){\small{Distance transitive graphs}}

\put(26,-25){\oval(27,22)}
\put(39,-28){\vector(1,0){60}}
\put(99,-28){\small{\textbf{Distance regular graphs}}}

\end{picture}
\label{fig:1}
\end{figure}
\vspace*{25mm}

\section{PBIBD(t)s from pattern
polynomial graphs}\label{sec:four}

A design is an ordered pair $(V,\mathcal{B})$ with point set $V$ and
set of blocks $\mathcal{B}$ such that $\mathcal{B}$ is a collection
of subsets of $V$.

A design  $(V,\mathcal{B})$ is called $t-(v,k,\la)$ design(some times only t-design) if $|V|=v, |B|=k\;\forall 
B\in \mathcal{B}$ and each subset of $V$ of cardinality $t$ is contained in $\la$ blocks.
One can show by counting that a $t$-design is an $i$-design for each
$0\leq i\leq t$. In fact a $t-(v,k,\la)$ design is a
$i-(v,k,\la_i)$ design with $\la_i=\la {v-i\choose i-1}/{k-i\choose
t-i}$ for each $0\leq i\leq t$.

A \textit{balanced incomplete block design}(BIBD) is a 2-design. The
parameters $\la_1$ and $\la_0$ are usually denoted by $r_1$
(replication number) and $b$ (number of blocks).

\begin{definition}
Given $v$symbols $1,2,\ldots,v$, a relation satisfying the
following condition is said to be an {\emph{symmetric association scheme}} with $m$
association classes:
\begin{enumerate}
\item Any two symbols $\alpha$ and $\beta$ are either
first,second,.. or mth associates and this relationship is
symmetrical. We denote $(\alpha,\beta)=i$, when $\alpha$ and
$\beta$ ith associates.
\item Every symbol $\alpha$ has $n_i$, $i$th associates, the
number $n_i$ being independent of $\alpha$.
\item If $(\alpha,\beta)=i$ the number of symbols $\gamma$ that
satisfy simultaneously $(\alpha,\gamma)=j$ and $(\beta,\gamma)=j'$
is $p^i_{jj'}$ and this number is independent of $\alpha$ and
$\beta$. Further $p^i_{jj'}=p^i_{j'j}$
\end{enumerate}
\end{definition}
The numbers $v, n_i,p^i_{jj'}$ are called the parameters of the
association scheme. If the relations are not symmetric, then it is
called an association scheme.

Let $R_i=\{(\alpha,\beta)|(\alpha,\beta)=i\}$ be the set of all
$i$-th associates. Then the relation $R_i$ of an  association scheme
can be described by a $0,1$-matrices $A_i$. Hence above definition
can be described in terms of  matrices as follows.

An association scheme with $d$ associate classes is a set
$\mathfrak{A}=\{A_0,\ldots, A_d\}$ of 0,1-matrices such that
\begin{enumerate}
 \item $A_0=I$.
\item $A_0+A_1+\dots +A_d=J$.
\item $A_i^T\in \mathfrak{A}$.
\item $A_iA_j=A_jA_i\in span(\mathfrak{A})$.
\end{enumerate}

If $A_i^T=A_i(1\leq i\leq d)$, then $\mathfrak{A}$ is a symmetric
association scheme also called Bose-Mesner algebra. For example, if
$X$ is a pattern polynomial graph, then $\A(X)$ is a Bose-Mesner
algebra. But every Bose-Mesner algebra can not be obtained in this
way. Now we will give an example of a Bose-Mesner algebra
$\mathfrak{A}$ which is not equal to adjacency algebra of any graph.

Let $G$ be a finite abelian group of order $n(>2)$. Each element of
$G$ gives rise to a permutation of $G$, the permutation corresponding
to '$a$ maps g in $G$ to $ga$'. Hence for each element $g$ in $G$ we
have a permutation matrix $P_g$; the map $g \rightarrow P_g$ is a
group homomorphism. Therefore $P_gP_h=P_hP_g, P(g^{-1}) = P_g^T$. We
have $P(1)=I$ and $\sum_{g\in G}P_g=J$. Hence the matrices $P_g$
forms an association scheme with $v-1$ classes.  Note the
association scheme obtained in this way is same as centralizer
algebra $\V_{\mathbb{C}}(G)$ which is further equal to group algebra
$\C[G]$. The restricted centralizer algebra $\V_r(G)$, consisting of
all real, symmetric matrices in $\V_{\mathbb{C}}(G)$ is a real
subalgebra of $\V_{\mathbb{C}}(G)$, which is closed under Hadamard
product and spanned by the matrices $P_g+P^T_g \;\forall g\in G$.
That is $\V_r(G)$ is a symmetric association scheme. Further if we
assume that  $G$ is an elementary abelian 2-group, then
$\C[G]=\V_r(G)=\V_{\mathbb{C}}(G)$. Consequently $\C[G]$ is a
Bose-Mesner algebra and $\dim(\C[G])=n$. Now from
Corollary~\ref{cor:mul} there exists no graph $X$ with $n>2$
vertices such that $\A(X)=\C[G]$.
\begin{definition}
Given an $m$-association scheme on $v$-symbols a PBIBD(m) with $m$
associate classes is defined as follows.
A PBIBD(m) with $m$ associate classes is an arrangement of $v$
symbols in $b$ sets of size $k(<v)$ such that
\begin{enumerate}
\item Every symbol occurs at most once in a set.
\item Every symbol occurs in $r$ sets.
\item Two symbols $\alpha$ and $\beta$ occur in $\lambda_i$ sets,
if $(\alpha,\beta)=i$ and $\lambda_i$ is independent of symbols
$\alpha$ and $\beta$.
\end{enumerate}
\end{definition}

The numbers $v, b, r, k,\lambda_i$ are the parameters of the PBIBD.
The PBIBD is usually identified by the association scheme of the
symbols. For more information on design theory  the reader is
referred to
 Raghavarao~\cite{rao}.  For any
PBIBD(t) we will write the parameters as
$(v,b,r_1,k_1,\lambda_1,\ldots ,\lambda_t)$ where $v$ is the number
of points , $b$ is the number of blocks, $r_1$ is called the replication number, $k_1$ is the number of elements in any block of the design.

If $N$ is the incidence matrix of a design $\mathcal{D}$, then we say that the design $\mathcal{D}$ is obtained from the  graph $X$ if $NN^T\in \A(X)$. For example, if $X$ is a pattern polynomial graph with $n$ vertices ,  then
\begin{enumerate}
\item Every BIBD with $n$ points is obtained from $X$. In fact every t-design with $n$ points and  $t\geq 2$ is obtained from $X$.
\item Let a graph  $Y$ be a polynomial in $X$  and  $\mathcal{D}_1=(V(Y),E(Y))$ be a design with points as vertices of graph $Y$ and blocks as edges of $Y$. Then $NN^T=D+A(Y)\in \A(X)$ where $N$ is the incidence matrix of design $\mathcal{D}_1$, which is also 0,1-incidence matrix of the graph $Y$ and $D$ is the diagonal matrix with diagonal entries are degree of vertices of $Y$. Note that $\mathcal{D}_1$ is a PBIBD(r), where $r$ is the degree of the minimal polynomial of $X$.
\end{enumerate}

Let $v$ be any vertex in a graph  $Z$, $N(v)$ be the set of vertices
which are adjacent to $v$ in $Z$, $\mathcal{B}_2=\{N(v)|v\in V(Z)\}$
and $\mathcal{B}_3=\{N(v)\cup\{v\}|v\in V(Z)\}$. If
$\mathcal{D}_2=(V(Z),\mathcal{B}_2)$ and
$\mathcal{D}_3=(V(Z),\mathcal{B}_3)$, then $\mathcal{D}_2$ and
$\mathcal{D}_3$  are designs on $V(Z)$ with
$n=|V(Z)|=|\mathcal{B}_2|=|\mathcal{B}_3|=b$. Hence  $A(Z)$ is the
incidence matrix of the designs $\mathcal{D}_2$ and $I+A(Z)$ is the
incidence matrix of $\mathcal{D}_3$. If $Z$ is a $k$-regular graph,
then all blocks in the design $\mathcal{D}_2(\mathcal{D}_3)$ are
$k$-subsets ($k+1$ subsets) of $V(Z)$. And  each vertex  belongs
exactly $k$ ($k+1$) blocks. In other words
$\mathcal{D}_2(\mathcal{D}_3)$ is a 1-design with $r_1=k_1$. Further
if we assume  $Z$ is  a polynomial in a pattern polynomial graph
$X$, then the designs $\mathcal{D}_2$ and $\mathcal{D}_3$  are
PBIBD(r)s where $r$ is the degree of minimal polynomial of $X$.
Hence we have the following result.

\begin{lemma}
Let a k-regular graph $Y$ be a polynomial in a pattern polynomial graph $X$
with $n$ vertices and $\mathcal{D}_i,i=2,3$ are designs on $V(Y)$  defined as above. Then $\mathcal{D}_2(\mathcal{D}_3)$
is a   PBIBD(r) obtained from $X$ with parameters
$(n,n,k,k,\lambda_1,\ldots ,\lambda_r)((n,n,k+1,k+1,\lambda^{'}_1,\ldots ,\lambda^{'}_r))$ for some $\la_i$ and $\la_i^{'}$.
\end{lemma}

\section{On the polynomial of a pattern polynomial graph} \label{sec:five}

A graph $Y$ is said to be a polynomial in a graph $X$ if $A(Y)\in
\A(X)$. For an arbitrary graph $X$ it seems difficult to find
whether a given graph is polynomial in $X$ or not. This question is answered for orbit poynomial graph and distance regular graphs by  [Robert
A.Beezer~\cite{R:B4}] and [Paul M.Weichsel~\cite{P:W1}] respectively. In the following lemma we generalize those results to all pattern polynomial graphs.
\begin{lemma}
 If $X$ is a pattern polynomial graph with standard basis
$\{P_1,P_2,\ldots ,P_r\}$ where $P_1=I$, then a graph $Y$ is a
polynomial in $X$ if and only if $A(Y)=\sum_{i=2}^{r-1}a_iP_i$ where
$a_i\in\{0,1\}$.
\end{lemma}

\begin{proof}
Direct consequence of Corollary~\ref{cor:sb}.
\end{proof}

\begin{cor}
There are  $2^{r-1}$ graphs in the adjacency algebra of a pattern
polynomial graph $X$, where $r$ is the degree of the minimal
polynomial of $A(X)$.
\end{cor}
Another trivial fact is as follows.
\begin{lemma}
Let a graph $Y$ be a polynomial in a pattern polynomial graph $X$,
then $\c(Y)\subseteq \c(X)$.
\end{lemma}

If a graph $Y$ is a polynomial in a pattern polynomial graph $X$,
then $\c(Y)$ is a symmetric (every matrix in $\c(Y)$ is symmetric)
commutative algebra. Hence 

\begin{enumerate}
\item $Y$ is a walk regular graph,
\item  $Y$ is a strongly distance-balanced graph, from Lemma~\ref{lem:dp},
\item  $Y$ has a multiple eigenvalue, whenever $Y\ne K_2$, from Corollary~\ref{cor:mul},
\item   $\dim(\c(Y))\leq n-1$, from Corollary~\ref{cor:mul}. Further if the number of vertices in $Y$ is odd, then $\dim(\c(Y))\leq \frac{n+1}{2}$.
\end{enumerate}

Now it is interesting to answer the following  question:
 If $Y$ is a graph
such that $\c(Y)$ is symmetric commutative algebra, then ``does there
exist a pattern polynomial graph $X$ such that $Y$ is a polynomial
in $X$?''.  For example, if $Y$ is a circulant graph (Cayley graph
on cyclic group) with $n$ vertices, then clearly $\c(Y)$ is
symmetric commutative algebra and it is also known that $Y$ is a
polynomial in cycle graph $C_n$, which is a pattern polynomial
graph.

If a graph $Y$ is a polynomial in a graph $X$, then there exists a unique polynomial $p_Y(x)\in\C[x]$, with  degree less than the degree of the minimal of $X$, such that $A(Y)=p_Y(A(X))$. It is  called representor polynomial of $Y$.
If $X$ is a pattern polynomial graph and $A(Y)=\sum_ia_iP_i$, then $p_Y=\sum_ia_ip_{X_{P_i}}$. If a graph $Y$ is a polynomial in $X$ with representor
polynomial  $p_Y(x)$, then the eigenvalues of $A(Y)$ are
$p_Y(\lambda_i)$, where $\lambda_i\;(0\leq i\leq n-1)$ are
eigenvalues of $A(X)$. 

The following result gives whether a graph $Y$ which is a polynomial in a graph $X$ is singular or not. Recall a graph is said to be singular  if its adjacency matrix is singular.

\begin{lemma}
Let a graph $Y$ be a polynomial in a graph $X$.
Then $Y$ is singular if and only if $\deg(\gcd(p(x),p_Y(x)))\geq 1$, where $p(x)$ is the minimal polynomial of $A(X)$ and  $p_Y(x)$ is the representor polynomial of $Y$ with respect to $X$.
\end{lemma}


\begin{thebibliography}{10}
\bibitem{R:B1} Robert A.Beezer, Orbit polynomial graphs of prime order,
 Discrete Mathematics 67 (1987) 139-147.

\bibitem{R:B4} Robert A.Beezer, A disrespectful polynomial,
Linear Algebra and its Applications,
 Volume 128, January 1990, Pages 139-146.

\bibitem{biggs}
N.L. Biggs,  Algebraic Graph Theory(second ed.), Cambridge University Press,
 Cambridge (1993).

\bibitem{B:R} A. E. Brouwer, A. M. Cohen, A. Neumaier,  Distance regular Graphs, Springer-Verlag, (1989).

\bibitem{dam} Damerell. R.M,  On Moore graphs, proc. Cambridge
Philos. sec.74,227-236.


\bibitem{davis} Philip J. Davis. Circulant matrices "A Wiley-interscience publications,(1979).

\bibitem{G:R}Chris D. Godsil $\&$ Gordon Royle. Algebraic Graph Theory, Springer-Verlag, (2001).

\bibitem{G:D1}Chris D. Godsil, Compact graphs and equitable partitions , Linear algebra and applications, 255:259-266 (1997).

\bibitem{G:M} C. D. Godsil and B. D. McKay, Feasibility conditions for the existence of walk-regular graphs, Linear algebra and its applications  30:51-61(1980).


\bibitem{H:K} Kenneth Hoffman and Ray Kunge, Linear Algebra (second edition), Prentice-Hall, (1971).

\bibitem{K:R:R:T} M.klin, C.R$\ddot{u}$cker,G.R$\ddot{u}$cker, G.Tinhofer,  Algebraic Combinatorics in Mathematical Chemistry Methods and Algorithms.I. Permutation Groups and Coherent (cellular) Algebras. Technical report Technische universit$\ddot{a}$t M$\ddot{u}$nchen, TUM-M9510(1995).


\bibitem{rao} Raghavarao, Constructions and combinatorial problems
in Design of Experiments, Wiley, New York (1971).

\bibitem {spence} E.Spence, Regular two-graphs on 36 vertices, Linear Alg. Appl. 226-228 (1995), 459-497.

\bibitem{S:P} Sergio Cabello and Primo$\check{z}$  Luk$\check{s}$i$\check{c}$,
The complexity of obtaining a distance-balanced graph, The electronic journal of combinatorics 18 (2011), P49.


\bibitem{P:W1}Paul M.Weichsel, On distance-regularity in graphs, Journal of combinatorial theory, Series B 32, 156-161 (1982).


\end{thebibliography}
\end{document}